\documentclass[12pt]{article}
\usepackage[latin1]{inputenc}
\usepackage{amsmath, amsfonts, amsthm,  amssymb, latexsym}
\usepackage{color,hyperref}
\usepackage[all]{xy}

\allowdisplaybreaks[4]

\newtheorem{theorem}{Theorem}[section]
\newtheorem{proposition}[theorem]{Proposition}
\newtheorem{corollary}[theorem]{Corollary}
\newtheorem{example}[theorem]{Example}

\newtheorem{remark}[theorem]{Remark}

\newtheorem{lemma}[theorem]{Lemma}
\newtheorem{final remark}[theorem]{Final Remark}
\newtheorem{definition}[theorem]{Definition}

\begin{document}

\title{Duality theory for generalized summing linear operators}
\author{Geraldo Botelho\thanks{Supported by CNPq Grant
304262/2018-8 and Fapemig Grant PPM-00450-17.\newline 2020 Mathematics Subject Classification: 47B10, 47L20, 46B45, 47B37, 46B10.\newline Keywords: Banach operator ideals, summing operators, vector-valued sequence spaces.} \, and  Jamilson R. Campos}
\date{}
\maketitle

\begin{abstract} Generalizing classical results of the theory of absolutely summing operators, in this paper we characterize the duals of a quite large class of Banach operator ideals defined or characterized by the transformation of vector-valued sequences.
\end{abstract}

\section{Introduction and background}

In the theory of operator ideals, classes of operators that are defined or characterized by the transformation of vector-valued sequences play a major role. The main reason is the striking success of the ideal of absolutely summing operators, which goes back to the works of Grothendieck \cite{grothen56}. A large amount of research has been devoted to this kind of operator ideals, we refer the reader to the classical monographs \cite{df,djt,pietschlivro} and for recent developments we refer to \cite{nacib, nacib1, angulo, bayart, bayart1, popa}.

In \cite{botelhocampos} we provided an unifying approach to this type of Banach operator ideals. The main tool is the notion of sequence classes, which encompasses the most used classes used in the theory, such as $p$-summable, weakly $p$-summable, almost inconditionally summable sequences, and many others (cf. Example \ref{exsc}).  This approach recovers several well studied ideals and gives rise to new classes, see, e.g., \cite{botelhocampossantos, diana}. The range of applications of this quite general framework has been recently expanded in the linear and nonlinear settings, see, e.g, \cite{botelhofreitas, joilsonFab, joilsonFab2}.

A central issue in the theory of operator ideals is the duality theory: knowing the description of the operators belonging to a given operator ideal $\cal I$, can one describe the operators whose adjoints belong to $\cal I$ and the adjoints of the operators belonging to $\cal I$? Cornerstones in this line of investigation are Schauder's Theorem for compact operators, Gantmacher's Theorem for weakly compact operators and the contributions of J. S. Cohen \cite{cohen73} and H. Apiola \cite{apiola}, who answered the questions above for the ideals of absolutely $p$-summing and absolutely $(p,q)$-summing operators.

Our aim in this paper is to start the duality theory for the operator ideals defined or characterized by the transformation of vector-valued sequences within the framework of sequence classes introduced in \cite{botelhocampos}. In Section 2 we define the dual $X^{\rm dual}$ of a sequence class $X$ and prove when $X^{\rm dual}(E') = X(E)'$. Given sequence classes $X$ and $Y$, in Section 3 we investigate the following implications: an operator is $(X;Y)$-summing if and only if its adjoint is $(Y^{\rm dual}; X^{\rm dual})$-summing; the adjoint of an operator is $(X;Y)$-summing if and only if the operator is $(Y^{\rm dual}; X^{\rm dual})$-summing.

The letters $E$ and $F$ shall denote Banach spaces over $\mathbb{K} = \mathbb{R}$ or $\mathbb{C}$. The closed unit ball of $E$ is denoted by $B_E$ and its topological dual by $E'$. The symbol $E \stackrel{1}{\hookrightarrow} F$ means that $E$ is a linear subspace of $F$ and $\|x\|_F \leq \|x\|_E$ for every $x \in E$. For $1 \le p < \infty$, the number $p^*$ is the conjugated index of $p$, that is, $1= 1/p + 1/{p^*}$. The symbol $x\cdot e_j$ indicates the sequence $(0,\ldots, 0,x,0, 0,\ldots )$, where $x$ appears at the $j$-th coordinate.

By ${\cal L}(E;F)$ we denote the Banach space of continuous linear operators $T \colon E \longrightarrow F$ endowed with the usual sup norm.
By $\widehat{T}$ we mean the induced operator
$$\widehat{T} \colon E^\mathbb{N} \longrightarrow F^\mathbb{N}~,~\widehat{T}((x_j)_{j=1}^\infty) = (T(x_j))_{j=1}^\infty.$$
Restrictions of $\widehat{T}$ to subspaces of $E^{\mathbb{N}}$ are still denoted by $\widehat{T}$. 

According to \cite{botelhocampos}, a {\it sequence class} is a rule $X$ that assigns to each Banach space $E$ a Banach space $X(E)$ of $E$-valued sequences, that is $X(E)$ is a vector subspace of $E^{\mathbb{N}}$ with the coordinatewise operations, such that:\\
(i) $c_{00}(E) \subseteq X(E) \stackrel{1}{\hookrightarrow}  \ell_\infty(E)$ for every Banach space $E$.\\
(ii) $\|x \cdot e_j\|_{X(\mathbb{E})}= \|x\|_E$  for all $E$, $x \in E$ and $j \in \mathbb{N}$.
		
The sequence class environment is an abstract tool designed to deal with classes of linear operators that improve the summability of sequences/the convergence of series: for sequence classes $X$ and $Y$, a linear operator $T \in \mathcal{L}(E;F)$ is called $(X;Y)$-summing if $(T(x_j))_{j=1}^\infty \in Y(F)$ whenever $(x_j)_{j=1}^\infty \in X(E)$. In this case we write $T \in \mathcal{L}_{X;Y}(E;F)$ and define
$$\|T\|_{X;Y} := \|\widehat{T}\colon X(E) \longrightarrow Y(F) \|_{\mathcal{L}(X(E); Y(F))}.$$
Conditions under which $(\mathcal{L}_{X;Y}, \|\cdot\|_{X;Y})$ is a Banach operator ideal are established in \cite{botelhocampos}.

We say that a sequence class $X$ is:\\
$\bullet$ \emph{linearly stable} if  $\mathcal{L}_{X;X}(E;F) = \mathcal{L}(E;F)$ isometrically for all Banach spaces $E$ and $F$.\\
$\bullet$ \emph{finitely determined} if for every sequence $(x_j)_{j=1}^\infty \in E^{\mathbb{N}}$, we have $(x_j)_{j=1}^\infty \in X(E)$ if and only if $\displaystyle\sup_k \left\|(x_j)_{j=1}^k  \right\|_{X(E)} < +\infty$ and, in this case,
$$\left\|(x_j)_{j=1}^\infty  \right\|_{X(E)} = \sup_k \left\|(x_j)_{j=1}^k  \right\|_{X(E)}. $$
$\bullet$ {\it finitely dominated} if one of the following conditions hold:\\
(I) There exists a finitely determined sequence class $Y$ such that, for every Banach space $E$, $X(E)$ is a closed subspace of $Y(E)$ and, for every sequence $(x_j)_{j=1}^\infty \in Y(E)$, $(x_j)_{j=1}^\infty \in X(E)$ if and only if $\displaystyle\lim_k \|(x_j)_{j=k}^\infty\|_{Y(E)} = 0$. In this case we write $X < Y$.\\
(II) There exists a finitely determined sequence class $Y$ such that, for every Banach space $E$, $X(E)$ is a closed subspace of $Y(E)$ and, for every sequence $(x_j)_{j=1}^\infty \in Y(E)$, $(x_j)_{j=1}^\infty \in X(E)$ if and only if $\displaystyle\lim_{k,l} \|(x_j)_{j=k}^l\|_{Y(E)} = 0$. In this case we write $X \prec Y$.

For the benefit of the reader, we list the most commonly used sequence classes.

\begin{example}\label{exsc}\rm
Letting $X(E)$ be any of the spaces listed below, the rule $E \mapsto X(E)$ is a linearly stable sequence class:

\medskip

\noindent $\bullet$ $\ell_\infty(E)$ = bounded $E$-valued sequences with the sup norm.\\
$\bullet$ $c_0(E)$ = norm null $E$-valued sequences with the sup norm.\\
$\bullet$ $c_0^w(E)$ = weakly null $E$-valued sequences with the sup norm.\\
$\bullet$ $\ell_p(E) $ = absolutely $p$-summable $E$-valued sequences with the usual norm $\|\cdot\|_p$.\\
$\bullet$ $\ell_p^w(E)$ = weakly $p$-summable $E$-valued sequences with the norm
$$\|(x_j)_{j=1}^\infty\|_{w,p} = \sup_{\varphi \in B_{E'}}\|(\varphi(x_j))_{j=1}^\infty\|_p. $$
$\bullet$ $\ell_p^u(E) = \left\{(x_j)_{j=1}^\infty \in \ell_p^w(E) : \displaystyle\lim_k \|(x_j)_{j=k}^\infty\|_{w,p} = 0 \right\}$ with the norm inherited from $\ell_p^w(E)$ (unconditionally $p$-summable sequences, see \cite[8.2]{df}).\\
$\bullet$ ${\rm Rad}(E)$ = almost unconditionally summable $E$-valued sequences, in the sense of \cite[Chapter 12]{djt}, with the norm $\displaystyle\|(x_j)_{j=1}^\infty\|_{\rm Rad(E)} = \left(\int_0^1 \left\|\sum_{j=1}^\infty r_j(t) x_j \right\|^2 dt \right)^{1/2},$
where $(r_j)_{j=1}^\infty$ are the Rademacher functions.\\
$\bullet$ $\displaystyle{{\rm RAD}(E)} = \left\{(x_j)_{j=1}^\infty \in E^{\mathbb{N}}: \|(x_j)_{j=1}^\infty\|_{{\rm RAD}(E)}:= \sup_k \|(x_j)_{j=1}^k\|_{{\rm Rad}(E)}< +\infty\right\}$ \cite{BotelhoBlasco, blascoetal}.\\
$\bullet$ $\displaystyle\ell_p\langle E \rangle  = \left\{(x_j)_{j=1}^\infty \in E^{\mathbb{N}}: \|(x_j)_{j=1}^\infty\|_{\ell_p\langle E \rangle}:= \sup_{(\varphi_j)_{j=1}^\infty \in B_{\ell_{p^*}^w(E')}} \|(\varphi_j(x_j))_{j=1}^\infty\|_1< +\infty\right\}$  (Cohen strongly $p$-summable sequences, see, e.g., \cite{cohen73}).\\
$\bullet$ $\displaystyle\ell_p^{\rm mid}(E)  =  \left\{(x_j)_{j=1}^\infty \in E^{\mathbb{N}}: \|(x_j)_{j=1}^\infty\|_{mid,p} := \sup_{(\varphi_n)_{n=1}^\infty \in B_{\ell_p^w(E')}} \left(\sum_{n=1}^\infty \sum_{j=1}^\infty |\varphi_n(x_j)|^p\right)^{1/p}< +\infty\right\}$ (mid $p$-summable sequences, see, e.g.,  \cite{botelhocampossantos}).

The sequence classes $\ell_\infty(\cdot), \ell_p(\cdot), \ell_p^w(\cdot),  \ell_p\langle \,\cdot\, \rangle, \ell_p^{\rm mid}(\cdot)$ and ${\rm RAD}(\cdot)$ are finitely determined, and the sequences classes $\ell_p^u(\cdot)$ and ${\rm Rad}(\cdot)$ are finitely dominated because $\ell_p^u(\cdot) < \ell_p^w(\cdot)$ and ${\rm Rad}(\cdot) \prec {\rm RAD}(\cdot)$.
\end{example}

\section{The dual of a sequence class}

Some preparation is needed to define the dual of a sequence class.

\begin{definition}\rm
	A sequence class $X$ is \emph{spherically complete} if $(\lambda_jx_j)_{j=1}^\infty \in X(E)$ and
 $\|(\lambda_jx_j)_{j=1}^\infty \|_{X(E)} = \|(x_j)_{j=1}^\infty\|_{X(E)}$ whenever $(x_j)_{j=1}^\infty \in X(E)$ and $(\lambda_j)_{j=1}^\infty \in \mathbb{K}^\mathbb{N}$, with $|\lambda_j| = 1$ for every $j$.
\end{definition}

\begin{lemma}\label{lemaserie}
	Let $X$ be a spherically complete sequence class and $(x_j)_{j=1}^\infty \in E^\mathbb{N}$. Then the following sentences are equivalent:\\
{\rm (a)} The series $\sum_{j=1}^{\infty} \varphi_j(x_j)$ converges for every $(\varphi_j)_{j=1}^\infty \in X(E')$.\\
{\rm (b)} The series $\sum_{j=1}^{\infty} |\varphi_j(x_j)|$ converges for every $(\varphi_j)_{j=1}^\infty \in X(E')$.\\
		In this case,
	\[\sup_{(\varphi_j)_{j=1}^\infty \in  B_{X(E')}}\left|\sum_{j=1}^{\infty} \varphi_j(x_j)\right| = \sup_{(\varphi_j)_{j=1}^\infty \in  B_{X(E')}}\sum_{j=1}^{\infty} |\varphi_j(x_j)|.\]
\end{lemma}

\begin{proof} The implication
	(b) $\Rightarrow$ (a) and the corresponding inequality are immediate. Let us prove (a) $\Rightarrow$ (b). Let $(\psi_j)_{j=1}^\infty$ be the sequence defined by
	\begin{displaymath}
	\psi_j = \left\{ \begin{array}{l}
	\varphi_j,\ \mathrm{if}\ \varphi_j(x_j) \geq 0\\
	- \varphi_j,\ \mathrm{if}\ \varphi_j(x_j) < 0,
	\end{array} \right.
	\end{displaymath}
	in the real case, and $\psi_j = \varphi_j e^{-i\theta_j}$ in the complex case, where $\theta_j$ is the principal argument of $\varphi_j(x_j)$.
	In both cases, as $X$ is spherically complete, $(\psi_j)_{j=1}^\infty \in X(E')$ and  $(\psi_j)_{j=1}^\infty \in B_{X(E')}$ if $(\varphi_j)_{j=1}^\infty \in B_{X(E')}$. So,
	\[\sum_{j=1}^{\infty} |\varphi_j(x_j)| = \sum_{j=1}^{\infty} \psi_j(x_j) \]
	and we obtain (b). Moreover,
	\[\sup_{(\varphi_j)_{j=1}^\infty \in  B_{X(E')}}\sum_{j=1}^{\infty} |\varphi_j(x_j)| = \sup_{(\psi_j)_{j=1}^\infty \in  B_{X(E')}}\sum_{j=1}^{\infty} \psi_j(x_j) \le \sup_{(\psi_j)_{j=1}^\infty \in  B_{X(E')}}\left|\sum_{j=1}^{\infty} \psi_j(x_j)\right|.\]
\end{proof}

Now we are ready to define the dual of a sequence class.

\begin{definition}\rm
	The {\it dual} of a sequence class $X$ is a rule that assigns to each  Banach space $E$ the following space of $E$-valued sequences:
	\begin{equation*}
		X^{\rm dual}(E) = \left\{(x_j)_{j=1}^\infty\ \mathrm{in\ } E: \sum_{j=1}^\infty \varphi_j(x_j)\ \mathrm{converges\ } \text{for every}\ (\varphi_j)_{j=1}^\infty\ \mathrm{in\ } X(E')\right\}.
	\end{equation*}
\end{definition}

We proceed to investigate under what conditions $X^{\rm dual}$ a sequence class.
It is immediate that $X^{\rm dual}(E)$ is a linear space of $E$-valued sequences with the usual coordinatewise operations and that $c_{00}(E) \subseteq X^{\rm dual}(E)$ for every $E$. 

\begin{proposition}
	If $X$ is a spherically complete sequence class, then the expression
	\begin{equation*}
	\|(x_j)_{j=1}^\infty\|_{X^{\rm dual}(E)} :=  \sup_{(\varphi_j)_{j=1}^\infty \in B_{X(E')}}   \sum_{j=1}^\infty |\varphi_j(x_j)|
	\end{equation*}
	defines a complete norm on $X^{\rm dual}(E)$ and
	\begin{equation}\label{ellinfty}
	 X^{\rm dual}(E) \stackrel{1}{\hookrightarrow} \ell_\infty(E)\ \text{ for every Banach space } E.
	\end{equation}
\end{proposition}

\begin{proof} Given $(x_j)_{j=1}^\infty \in  X^{\rm dual}(E)$, as $X$ is spherically complete, thanks to Lemma \ref{lemaserie} the series in the definition of  $X^{\rm dual}(E)$ can replaced with the series $\sum_{j=1}^\infty |\varphi_j(x_j)|$. A standard closed graph argument gives the continuity of the operator
\begin{align*}
T_{(x_j)_{j=1}^\infty}\colon X(E') \longrightarrow \ell_1~,~ T_{(x_j)_{j=1}^\infty}((\varphi_j)_{j=1}^\infty) = (\varphi_j(x_j))_{j=1}^\infty.
\end{align*}
Thus, the supremum in the definition of the dual norm is finite and the norm axioms follow easily. Since $c_{00}(E') \subseteq X(E')$ and $\|\varphi\| = \|\varphi\cdot e_j\|_{X(E')}$ for all $\varphi \in E'$ and $j \in \mathbb{N}$,
\[ \|x_j\| = \sup_{\varphi \in B_{E'}} |\varphi(x_j)| \le \sup_{(\varphi_j)_{j=1}^\infty \in B_{X(E')}}   \sum_{j=1}^\infty |\varphi_j(x_j)| = \|(x_j)_{j=1}^\infty\|_{X^{\rm dual}(E)}, \]
therefore $X^{\rm dual}(E) \stackrel{1}{\hookrightarrow} \ell_\infty(E)$. The completeness of $X^{\rm dual}(E)$ follows by a straightforward argument using the completeness of $E$ and the fact, guaranteed by \eqref{ellinfty}, that convergence in $X^{\rm dual}(E)$ implies coordinatewise convergence.
\end{proof}

\begin{proposition}\label{proporp}
Let $X$ be a spherically complete sequence class. Then:\\
{\rm (a)} $X^{\rm dual}$ is a finitely determined and spherically complete sequence class.\\
{\rm (b)} If $X$ is linearly stable, then so is $X^{\rm dual}$.
\end{proposition}
\begin{proof} (a) Let $j \in \mathbb{N}$, $x \in E$ and $\varphi \in E'$ be given. As $x \cdot e_j \in X^{\rm dual}(E)$ and $\varphi \cdot e_j \in X(E')$, writing $(x_k)_{k=1}^\infty = x\cdot e_j$ and $(\varphi_k)_{k=1}^\infty = \varphi \cdot e_j$, we have $x_j=x$, $\varphi_j = \varphi$ and $x_k = 0$ and $\varphi_k = 0$ for $k \neq j$. Thus
	\begin{equation*} \|x\| = \sup_{\varphi \in B_{E'}}|\varphi(x)| \le  \sup_{(\varphi_k)_{k=1}^\infty \in B_{X(E')}}\sum_{k=1}^\infty|\varphi_k(x_k)| = \|(x_k)_{k=1}^\infty\|_{X^{\rm dual}(E)} = \|x \cdot e_j \|_{X^{\rm dual}(E)}.
	\end{equation*}
	On the other hand,
	\begin{align*}
	\|x \cdot e_j \|_{X^{\rm dual}(E)} & = \|(x_k)_{k=1}^\infty\|_{X^{\rm dual}(E)}
	= \sup_{(\varphi_k)_{k=1}^\infty \in B_{X(E')}}\sum_{k=1}^\infty|\varphi_k(x_k)| \\
	& \leq   \sup_{(\varphi_k)_{k=1}^\infty \in B_{X(E')}}\sum_{k=1}^\infty\|\varphi_k\|\cdot \|x_k\|  =\sup_{(\varphi_k)_{k=1}^\infty \in B_{X(E')}} \|\varphi_j\|\cdot \|x\| = \|x\|,
	\end{align*} proving that $\|x \cdot e_j \|_{X^{\rm dual}(E)} = \|x\|_E$. This was all that was left to prove that $X^{\rm dual}$ is a sequence class. For a sequence $(x_j)_{j='}^\infty \in E^{\mathbb{N}}$,
	\begin{align*}
	\sup_k \|(x_j)_{j=1}^k\|_{X^{\rm dual}(E)} & = \sup_k \sup_{(\varphi_j)_{j=1}^\infty \in B_{X(E')}} \sum_{j=1}^{k}\left\vert \varphi_j(x_j)\right\vert \\
	& = \sup_{(\varphi_j)_{j=1}^\infty \in B_{X(E')}}  \sup_k \sum_{j=1}^{k}\left\vert \varphi_j(x_j)\right\vert = \|(x_j)_{j=1}^\infty\|_{X^{\rm dual}(E)},
	\end{align*}
from which it follows that $X^{\rm dual}$ is finitely determined.

Let $(x_j)_{j=1}^\infty \in X^{\rm dual}(E)$, $(\lambda_j)_{j=1}^\infty \in \mathbb{K}^\mathbb{N}$ with $|\lambda_j| = 1$ for every $j \in \mathbb{N}$, and $(\varphi_j)_{j=1}^\infty \in X(E')$ be given. Since $(\lambda_j\varphi_j)_{j=1}^\infty \in X(E')$,
	\[
	\sum_{j=1}^{\infty}\left\vert \varphi_j(\lambda_jx_j)\right\vert = \sum_{j=1}^{\infty}\left\vert (\lambda_j\varphi_j)(x_j)\right\vert  < \infty,
	\]
	what gives us that $(\lambda_jx_j)_{j=1}^\infty \in X^{\rm dual}(E)$. And since $\|(\lambda_j\varphi_j)_{j=1}^\infty\|_{X(E')} = \|(\varphi_j)_{j=1}^\infty\|_{X(E')}$,
	\begin{align*}
	\|(\lambda_jx_j)_{j=1}^\infty\|_{X^{\rm dual}(E)} & = \sup_{(\varphi_j)_{j=1}^\infty \in B_{X(E')}} \sum_{j=1}^{n}\left\vert \varphi_j(\lambda_jx_j)\right\vert\\
	&=  \sup_{(\varphi_j)_{j=1}^\infty \in B_{X(E')}} \sum_{j=1}^{n}\left\vert (\lambda_j\varphi_j)(x_j)\right\vert = \|(x_j)_{j=1}^\infty\|_{X^{\rm dual}(E)}.
	\end{align*}
	So, $X^{\rm dual}$ is spherically complete.

\noindent (b) Let $T \in {\cal L}(E;F)$ and $(\varphi_j)_{j=1}^\infty \in X(F')$ be given. By the linear stability of $X$ we have  $(T'(\varphi_j))_{j=1}^\infty = (\varphi_j \circ T)_{j=1}^\infty\in X(E')$ and
\[ \left\Vert \widehat{T'} ((\varphi_j)_{j=1}^\infty)\right\Vert_{X(E')} = \left\Vert \left(T' (\varphi_j)\right)_{j=1}^\infty\right\Vert_{X(E')} \le \left\Vert T'\right\Vert\cdot\|(\varphi_j)_{j=1}^\infty\|_{X(F')},\]
what gives $\|(\varphi_j \circ T)_{j=1}^\infty\|_{X(E')} \le \|T\|\cdot \|(\varphi_j)_{j=1}^\infty\|_{X(F')}$. So, for every $(x_j)_{j=1}^\infty \in X^{\rm dual}(E)$,
\begin{align*}
\left\Vert \widehat{T}((x_j)_{j=1}^\infty)\right\Vert_{X^{\rm dual}(F)} & = \left\Vert \left(T(x_j)\right)_{j=1}^\infty\right\Vert_{X^{\rm dual}(F)} = \sup_{(\varphi_j)_{j=1}^\infty \in  B_{X(F')}} \sum_{j=1}^\infty \left\vert  \varphi_j (T(x_j)) \right\vert \\
& = \|T\|\cdot \sup_{(\varphi_j)_{j=1}^\infty \in  B_{X(F')}}  \sum_{j=1}^\infty \left\vert \left(\frac{\varphi_j \circ T}{\|T\|}\right) (x_j) \right\vert \\
& \le \|T\|\cdot \sup_{(\psi_j)_{j=1}^\infty \in  B_{X(E')}}  \sum_{j=1}^\infty \left\vert \psi_j(x_j) \right\vert = \|T\| \cdot \left\Vert (x_j)_{j=1}^\infty\right\Vert_{X^{\rm dual}(E)}.
\end{align*}
Now the equality $\left\Vert\widehat{T}\right\Vert = \left\Vert T\right\Vert$ follows easily. 
\end{proof}

\begin{example}\label{exdual}\rm We give some examples of sequence classes that fit in this framework and some of its duals. We shall return to some of these examples later.

The notation $\ell_p(\cdot)^{\rm dual}(\cdot)$ is quite cumbersome, so we shall write $\ell_p^{\rm dual}(\cdot)$ instead. Accordingly, we write $(\ell_p^w)^{\rm dual}(\cdot)$, $(\ell_p^u)^{\rm dual}(\cdot)$ and so on.\\
 (a) All sequence classes in Example \ref{exsc}, but  ${\rm Rad}(\cdot)$ and ${\rm RAD}(\cdot)$, are spherically complete.\\
 (b) For any Banach space $E$ and $1 \le p < \infty$, \[\ell_p^{\rm dual}(E) = \left\{(x_j)_{j=1}^\infty\ \mathrm{in\ } E: \sum_{j=1}^\infty |\varphi_j(x_j)|< \infty  \text{~for every}\ (\varphi_j)_{j=1}^\infty\in  \ell_p(E')\right\}.\]
 To avoid unnecessary computations, we wait until Example \ref{exdual2} to establish that, as expected, $\ell_p^{\rm dual}(\cdot) = \ell_{p^*}(\cdot)$. \\
  (c) Again as expected, let us check that $\ell_1^{\rm dual}(\cdot) = \ell_\infty(\cdot)$. Given a Banach space $E$, $(x_j)_{j=1}^\infty \in \ell_1^{\rm dual}(E)$, $k \in \mathbb{N}$ and $\varphi \in B_{E'}$, it is clear that the sequence $\varphi \cdot e_k =: (\widetilde{\varphi}_j)_{j=1}^\infty$ belongs to $B_{\ell_1(E')}$. So,
  \begin{align*} \|x_k\|& = \sup_{\varphi \in B_{E'}}|\varphi(x_k)| = \sup_{ \varphi \in B_{E'}} \sum_{j=1}^\infty |\widetilde{\varphi}_j(x_j)|\leq \sup_{(\varphi_j)_{j=1}^\infty \in B_{\ell_1(E')}} \sum_{j=1}^\infty |{\varphi}_j(x_j)| = \|(x_j)_{j=1}^\infty\|_{\ell_1^{\rm dual}(E)}.
  \end{align*}
  Thus $(x_j)_{j=1}^\infty \in \ell_\infty(E)$ and the underlying norm inequality follows. The reverse inclusion/norm inequality is clear.
  \\
 (d) Let us check that, unexpectedly,  $\ell_\infty^{\rm dual}(\cdot) = \ell_1(\cdot)$. Given a Banach space $E$ and $(x_j)_{j=1}^\infty \in \ell_\infty^{\rm dual}(E)$, we have $\sum_{j=1}^\infty \varphi_j(x_j)$ convergent for every $(\varphi_j)_{j=1}^\infty \in \ell_\infty(E')$. For each $j \in \mathbb{N}$, take $\varphi_j \in E'$ such that $\varphi_j(x_j) = \|x_j\|$ and $\|\varphi_j\| = 1$. Then $(\varphi_j)_{j=1}^\infty \in \ell_\infty(E')$, so
	$$\sum_{j=1}^\infty \|x_j\| = \sum_{j=1}^\infty \varphi_j(x_j) < \infty, $$
	proving that $(x_j)_{j=1}^\infty  \in \ell_1(E)$. Conversely, if $(x_j)_{j=1}^\infty \in \ell_1(E)$, then for every $(\varphi_j)_{j=1}^\infty \in  \ell_\infty(E')$,
	\begin{align*}
	\sum_{j=1}^\infty |\varphi_j(x_j)| & \leq \sum_{j=1}^\infty \|\varphi_j\|\cdot \|x_j\| \leq \left(\sup_k \|\varphi_k\|\right) \cdot \sum_{j=1}^\infty \|x_j\|\\
	& = \|(\varphi_j)_{j=1}^\infty\|_{\ell_\infty(E')} \cdot \|(x_j)_{j=1}^\infty\|_{\ell_1(E)} < \infty.
	\end{align*}
	This proves that $\sum_{j=1}^\infty \varphi_j(x_j)$ converges for every $(\varphi_j)_{j=1}^\infty \in  \ell_\infty(E')$, that is, $(x_j)_{j=1}^\infty \in \ell_\infty^{\rm dual}(E)$. The norm equality $\|\cdot\|_{\ell_\infty^{\rm dual}(E)}= \|\cdot\|_{\ell_1(E)}$ is immediate.\\
(e) It follows immediately from the respective definitions that $(\ell_p^w)^{\rm dual}(\cdot) = \ell_{p^*}\langle \cdot \rangle$.\\
 (f) Our dual procedure gives new sequence classes. For instance, we can define the (finitely determined, spherically complete and linearly stable) sequence class $\ell_p^{\rm mid}\left\langle \cdot \right\rangle$ by
	\begin{align*}
	\ell_p^{\rm mid}\left\langle E \right\rangle &:= (\ell_p^{\rm mid})^{\rm dual}(E) \\
	& = \left\{(x_j)_{j=1}^\infty\ \mathrm{in\ } E: \sum_{j=1}^\infty |\varphi_j(x_j)|< \infty\  \text{for every}\ (\varphi_j)_{j=1}^\infty\ \mathrm{in\ } \ell_p^{\rm mid}(E')\right\}.
	\end{align*}
	It is not difficult to prove that $\ell_p\langle E \rangle \stackrel{1}{\hookrightarrow} \ell_p^{\rm mid}\left\langle E \right\rangle \stackrel{1}{\hookrightarrow} \ell_p(E)$ for any Banach space $E$ and every $1<p<\infty$. We know that $\ell_1\langle E \rangle = \ell_1(E)$ and therefore $\ell_1\langle E \rangle = \ell_1^{\rm mid}\left\langle E \right\rangle = \ell_1(E)$. 	
\end{example}

Our next aim is to prove the duality $X^{\rm dual}(E') =X(E)'$. Some preparatory work in order.

\begin{definition} \rm A sequence class $X$ is said to be {\it finitely injective} if \begin{equation}\label{eqfin}
\|(x_j)_{j=1}^k\|_{X(E)} \leq \|(i(x_j))_{j=1}^k\|_{X(F)}
\end{equation}
whenever $i \colon E \longrightarrow F$ is a metric injection, $k \in \mathbb{N}$ and $x_1, \ldots, x_k \in E$.
\end{definition}

It is clear that if $X$ is also linearly stable, then we have an equality in \eqref{eqfin}.

\begin{example}\rm
	The sequence classes $c_0(\cdot), \ell_\infty(\cdot),\ell_p(\cdot), \ell_p^u(\cdot)$ and $\ell_p^w(\cdot)$ are finitely injective. 
\end{example}

The following elementary lemma will be helpful soon.

\begin{lemma}\label{lemaiso}\rm
	Let $J\colon A \longrightarrow B$ and $I\colon B \longrightarrow A$ be maps between linear spaces such that $J$ is linear and injective and $I|_{J(A)} = J^{-1}$. Then $J$ is surjective if and only if $I$ is injective.
\end{lemma}

\begin{theorem} \label{dualproof}
Let $E$ be a Banach space and $X$ be a linearly stable, finitely dominated and spherically complete sequence class. Then\\
{\rm (a)} The map 
\begin{align}\label{opdual} J\colon X^{\rm dual}(E') \longrightarrow X(E)'~,~
   J\left((\varphi_j)_{j=1}^\infty\right)\left((x_j)_{j=1}^\infty\right) = \sum_{j=1}^\infty \varphi_j(x_j),
\end{align}
is a well defined, injective continuous linear operator.

Suppose that, in addition, $X$ is finitely injective. Then:\\
{\rm (b)} $J$ is an isometric isomorphism from $X^{\rm dual}(E')$ onto a complemented subspace of $X(E)'$.\\
{\rm (c)} $J$ is an isometric isomorphism from $X^{\rm dual}(E')$ onto $X(E)'$ if and only if $c_{00}(E)$ is dense in $X(E)$.
\end{theorem}

\begin{proof}
(a) Given $(\varphi_j)_{j=1}^\infty \in  X^{\rm dual}(E')$, the series $\sum_{j=1}^\infty \psi_j(\varphi_j)$ converges for every $(\psi_j)_{j=1}^\infty \in X(E'')$. As $X$ is linearly stable, considering the canonical embedding $J_E \colon E \longrightarrow E''$, for any $(x_j)_{j=1}^\infty \in X(E)$ we have $(J_E(x_j))_{j=1}^\infty \in X(E'')$ and, in particular,
\[ \sum_{j=1}^\infty \varphi_j(x_j) = \sum_{j=1}^\infty J_E(x_j)(\varphi_j) \ \mathrm{converges}. \]
So, the operator $J\left((\varphi_j)_{j=1}^\infty\right)$ is well-defined, and its linearity is obvious. Calling on the linear stability of $X$ once again, we have
\begin{align*}
\left\vert\sum_{j=1}^{n}\varphi_j(x_j)\right\vert &= \left\Vert(x_j)_{j=1}^{\infty}\right\Vert_{X(E)}\cdot \sum_{j=1}^{n}\left\vert \varphi_j\left(\frac{x_j}{\left\Vert(x_j)_{j=1}^{\infty}\right\Vert_{X(E)}}\right)\right\vert \nonumber \\
&\leq \left\Vert(x_j)_{i=1}^{\infty}\right\Vert_{X(E)} \cdot \sup_{(y_j)_{i=1}^{\infty}\in B_{X(E)}} \sum_{j=1}^{n}\left\vert\varphi_j(y_j)\right\vert  \\
&= \left\Vert(x_j)_{i=1}^{\infty}\right\Vert_{X(E)} \cdot \sup_{(y_j)_{i=1}^{\infty}\in B_{X(E)}} \sum_{j=1}^{n}\left\vert J_E(y_j)(\varphi_j)\right\vert  \\
&\le \left\Vert(x_j)_{i=1}^{\infty}\right\Vert_{X(E)}\cdot \sup_{(\psi_j)_{i=1}^{\infty}\in B_{X(E'')}} \sum_{j=1}^{n}\left\vert \psi_j(\varphi_j)\right\vert \\&= \left\Vert(x_j)_{j=1}^{\infty}\right\Vert_{X(E)}\cdot \left\Vert(\varphi_j)_{j=1}^{n}\right\Vert_{X^{\rm dual}(E')}\nonumber
\end{align*}
for all $n \in \mathbb{N}$ and $0 \neq (x_j)_{j=1}^{\infty}\in X(E)$. Since $X^{\rm dual}$ is finitely determined, it follows that $J((\varphi_j)_{j=1}^\infty)$ is continuous. Therefore the operator $J$ is well-defined, linear (obvious) and continuous with $\|J\| \le 1$. The injectivity of $J$ follows easily.

\medskip

\noindent (b) Let $I_j\colon E\longrightarrow X(E)$ the mapping given by $I_j(x)=x \cdot e_j$ and $I\colon X(E)' \longrightarrow X^{\rm dual}(E')$ the operator given by $I(\varphi)=(\varphi\circ I_j)_{j=1}^{\infty}$. Let us prove that $I$ is well-defined. Given $(\psi_j)_{j=1}^{\infty}\in X(E'')$, for $n\in\mathbb{N}$ consider $M= {\rm span}\{\psi_1,\ldots,\psi_n\}$, $N={\rm span}\{\varphi \circ I_1,\ldots,\varphi \circ I_n\}$ and the identity operator $\text{id}_{E''} \in {\cal L}(E'';E'')$. Using the Weak Principle of Local Reflexivity \cite[p.\,73]{df}, for every $\delta>0$ there is an operator $S\in\mathcal{L}(M,E)$ such that
\begin{equation*}
\left\Vert S\right\Vert\leq(1+\delta)\left\Vert\mathrm{id}_{E^{\prime\prime}}\big\rvert_{M}\right\Vert=1+\delta \mathrm{~and}$$
$$
\psi_j(\varphi \circ I_j) = (\varphi \circ I_j)(S(\psi_j)),
\end{equation*} for every $j=1,\ldots,n$. 
Taking $\lambda_j = e^{-i\theta_j}$, where $\theta_j$ is the principal argument of $(\varphi \circ I_j)(S(\psi_j))$, using the assumptions on $X$ we get
\begin{align*}
\sum_{j=1}^{n}\left\vert\psi_j(\varphi \circ I_j)\right\vert &=  \sum_{j=1}^{n}\left\vert(\varphi \circ I_j)(S(\psi_j))\right\vert
 = \sum_{j=1}^{n}(\varphi \circ I_j)(S(\psi_j))\lambda_j\\
& = \sum_{j=1}^{n}(\varphi \circ I_j)(S(\lambda_j\psi_j)) = \sum_{j=1}^{n}\varphi(S(\lambda_j\psi_j)\cdot e_j) \\
& = \varphi((S(\lambda_j\psi_j))_{j=1}^{n})  \leq \left\Vert\varphi\right\Vert \cdot \left\Vert S\right\Vert \cdot \left\Vert(\lambda_j\psi_j)_{j=1}^{n}\right\Vert_{X(M)} \\
& = \left\Vert\varphi\right\Vert \cdot \left\Vert S\right\Vert \cdot \left\Vert(\lambda_j\psi_j)_{j=1}^{n}\right\Vert_{X(E'')}  \le \left\Vert\varphi\right\Vert (1+ \delta) \left\Vert(\psi_j)_{j=1}^{n}\right\Vert_{X(E'')}.
\end{align*}
Now, taking the supremum on $n$, making $\delta \longrightarrow 0$ and using that $X$ is finitely dominated, we obtain
\begin{equation*} 
\sum_{j=1}^{\infty}\left\vert\psi_j(\varphi\circ I_j)\right\vert \leq \left\Vert\varphi\right\Vert \cdot \left\Vert(\psi_j)_{j=1}^{\infty}\right\Vert_{X(E'')}
\end{equation*}
for every $(\psi_j)_{j=1}^{\infty}\in X(E'')$, from which we conclude that $I$ is well-defined. Clearly $I$ is linear and the inequality immediately above gives
$\left\Vert I(\varphi)\right\Vert_{X^{\rm dual}(E')}\leq \left\Vert\varphi\right\Vert$,
from which the continuity of $I$ follows. From $J\circ I=\mathrm{id}_{J(X^{\rm dual}(E'))}$ and
\begin{equation*}
\left\Vert\varphi\right\Vert = \left\Vert J\circ I(\varphi)\right\Vert = \left\Vert J(I(\varphi))\right\Vert \leq \left\Vert I(\varphi)\right\Vert_{X^{\rm dual}(E')} \leq \left\Vert\varphi\right\Vert,
\end{equation*} we conclude that  $J(X^{\rm dual}(E'))$ is a complemented subspace of  $X(E)'$ isometrically isomorphic to $X^{\rm dual}(E')$.

\medskip
\noindent (c) In both implications we use the well known consequence of the Hahn-Banach Theorem that a linear subspace $M$ of a normed space $E$ is dense if and only if the only functional $\varphi \in E'$ such that $\varphi|_{M} = 0$ is $\varphi = 0$.

Suppose that $J\colon X^{\rm dual}(E') \longrightarrow X(E)'$ is surjective. So, given a functional $\varphi \neq 0$ in $X(E)'$ there is a non-zero sequence $(\varphi_j)_{j=1}^\infty \in X^{\rm dual}(E')$ such that $J((\varphi_j)_{j=1}^\infty) = \varphi$. Let $0 \neq x \in E$ and $j_0 \in \mathbb{N}$ be such that $\varphi_j(x) \neq 0$. Thus, writing $(x_k)_{k=1}^\infty = x\cdot e_j$, we have
\[ \varphi(x \cdot e_j) = \sum_{j=1}^\infty \varphi_k(x_k) = \varphi_j(x) \neq 0 ,  \]
 therefore $\varphi|_{c_{00}(E)} \neq 0$, proving that the only functional in $X(E)'$ that vanishes in $c_{00}(E)$ is $\varphi = 0$. Conversely, by Lemma \ref{lemaiso} it is enough to check that the map $I$ defined above is injective if  $c_{00}(E)$ is dense in $X(E)$. To do so, let $\varphi \in X(E)'$ be such that $I(\varphi) = 0$. Thus
\begin{align*}
(\varphi \circ I_j)_{j=1}^\infty =0  & \Longrightarrow  \varphi \circ I_j(x) = 0 {\rm ~for~all~} x \in E {\rm ~and~} j \in \mathbb{N}\\
& \Longrightarrow  \varphi(x \cdot e_j) = 0 {\rm ~for~all~} x \in E {\rm ~and~} j \in \mathbb{N} \Longrightarrow \varphi|_{c_{00}(E)} = 0.
\end{align*}
The denseness $c_{00}(E)$ in $X(E)$ implies that $\varphi = 0$.
\end{proof}

\medskip

Naturally enough, from now on a finitely dominated, linearly stable, finitely injective and spherically complete sequence class $X$ such that $c_{00}(E)$ is dense in $X(E)$ for every $E$ shall be referred to as a {\it dual-representable sequence class}.

\begin{example}\label{exdual2}\rm
	Here we show that Theorem \ref{dualproof} recovers well known dualities and provides new ones. \\
(a) For $1 < p <\infty$, the sequence class $\ell_p(\cdot)$ is dual-representable, so
\[\ell_p^{\rm dual}(E') = (\ell_p(E))' = \ell_{p^*}(E')\] isometrically for any Banach space $E$, where the first equality follows from Theorem \ref{dualproof} and the second is well known. Interchanging $p$ and $p*$ in the second equality above, applying the Hahn-Banach Theorem and bearing Example \ref{exdual}(b) in mind, we have
\[\|(x_j)_{j=1}^\infty\|_{p^*} = \sup_{(\varphi_j)_{j=1}^\infty \in  B_{\ell_p(E')}} \sum_{j=1}^\infty |\varphi_j(x_j)| = \|(x_j)_{j=1}^\infty\|_{\ell_p^{\rm dual}(E)}\]
for every $(x_j)_{j=1}^\infty \in E^{\mathbb{N}}$. This shows that, as expected,  $\ell_p^{\rm dual}(\cdot) = \ell_{p^*}(\cdot)$.\\
(b) The formula $\ell_1^{\rm dual}(\cdot) = \ell_\infty(\cdot)$ proved in Example \ref{exdual}(c) can be recovered for dual spaces: since the sequence class $\ell_1(\cdot)$ is dual-representable,
\[\ell_1^{\rm dual}(E') = (\ell_1(E))' = \ell_\infty(E')\] isometrically for any Banach space $E$, where the first equality follows from Theorem \ref{dualproof} and the second is well known.\\
\noindent (c) For $1 \le p < \infty$, the sequence class $\ell_p^u(\cdot)$ is dual-representable, so
\[(\ell_p^u)^{\rm dual}(E') = (\ell_p^u(E))' = \ell_{p^*}\left\langle E'\right\rangle = (\ell_p^w)^{\rm dual}(E')\] isometrically for any Banach space $E$, where the first equality follows from Theorem \ref{dualproof}, the second can be proved by standard arguments (see \cite{ apiola, bu2005}), and the third comes from Example \ref{exdual}(e). We do not know if $(\ell_p^u)^{\rm dual}(\cdot) = (\ell_p^w)^{\rm dual}(\cdot)$. \\
\noindent (d) The sequence class $\ell_p^w(\cdot)$,  $1 \le p < \infty$, is not dual-representable, but from Theorem \ref{dualproof}(b) and Example \ref{exdual}(e) we conclude that $\ell_{p^*}\left\langle E'\right\rangle = (\ell_p^w)^{\rm dual}(E')$ is a complemented subspace of $\ell_p^w(E)'$, for every Banach space $E$.
\end{example}

\section{Adjoints and second adjoints}

Let $T' \in {\cal L}(F';E')$ be the adjoint of the operator $T \in {\cal L}(E;F)$ and let $X,Y$ be sequence classes. The main purpose of this section is to establish when the implications below hold with the corresponding norm inequalities:
\begin{equation*} T \in {\cal L}_{X;Y}(E;F) \Longleftrightarrow T' \in {\cal L}_{Y^{\rm dual};X^{\rm dual}}(F';E'), \end{equation*}
\begin{equation}\label{fequi}  T' \in {\cal L}_{X;Y}(F';E') \Longleftrightarrow T \in {\cal L}_{Y^{\rm dual};X^{\rm dual}}(E;F) .
\end{equation}

These implications will be proved with the help of Theorem \ref{dualproof}. A consequence regarding the equivalence $T \in {\cal L}_{X;Y}(E;F) \Longleftrightarrow T'' \in {\cal L}_{X;Y}(E'';F'')$ and illustrative examples shall be provided.

Let us proceed to establish the first implications. Given a sequence class $X$, since $X^{\rm dual}$ is a sequence class itself, we can compute its dual $(X^{\rm dual})^{\rm dual}$: for a sequence $(x_j)_{j=1}^\infty$ in a Banach space $E$,
\begin{align*} (x_j)_{j=1}^\infty \in \left(X^{\rm dual}\right)^{\rm dual}(E) & \Longleftrightarrow  \sum_{j=1}^\infty x_j'(x_j)\  \text{converges~for~every~} (x_j')_{j=1}^\infty \in X^{\rm dual}(E')\\ &\Longleftrightarrow
\sum_{j=1}^\infty x_j'(x_j) \text{~converges~for~every~} (x_j')_{j=1}^\infty \subseteq E' \text{~such that} \\
&~~~~~~~\sum_{j=1}^\infty x_j''(x_j')\  \text{converges~for~every~} (x_j'')_{j=1}^\infty \in X(E'').
\end{align*}

In the same fashion of the canonical embedding $E \stackrel{1}{\hookrightarrow} E''$, we have the following:

\begin{proposition}\label{canemb}
If $X$ is a linearly stable and spherically complete sequence class, then $X(E) \stackrel{1}{\hookrightarrow} \left(X^{\rm dual}\right)^{\rm dual}(E)$ for every Banach space $E$.
\end{proposition}
\begin{proof}
Let $(x_j)_{j=1}^\infty \in X(E)$ be given. Consider $(x_j')_{j=1}^\infty \text{ in } E'$ such that $\sum\limits_{j=1}^\infty x_j''(x_j')$ converges for every $(x_j'')_{j=1}^\infty \in X(E'')$. Since $X$ is linearly stable, we have $(J_E(x_j))_{j=1}^\infty \in X(E'')$, so
$$\sum_{j=1}^\infty x_j'(x_j) = \sum_{j=1}^\infty J_E(x_j)(x_j') {\rm ~converges}, $$
proving that $(x_j)_{j=1}^\infty \in \left(X^{\rm dual}\right)^{\rm dual}(E)$. For the norm inequality, let $(x_j')_{j=1}^\infty \text{ in } E'$ be such that $\sum\limits_{j=1}^\infty |x_j''(x_j')| < \infty$ for every $(x_j'')_{j=1}^\infty \in X(E'')$ and $\sup\limits_{(x_j'')_{j=1}^\infty \in B_{X(E'')}}\sum\limits_{j=1}^\infty |x_j''(x_j')| \leq 1$. Since $(J_E(x_j))_{j=1}^\infty \in X(E'')$, we have $\displaystyle\left(\frac{J_E(x_j)}{\|(J_E(x_j))_{j=1}^\infty\|_{X(E'')}} \right)_{j=1}^\infty \in B_{X(E'')}$ and
\begin{equation}\label{eq}\sum_{j=1}^\infty \left|\frac{J_E(x_j)}{\|(J_E(x_j))_{j=1}^\infty\|_{X(E'')}}(x_j') \right| \leq  \sup\limits_{(x_j'')_{j=1}^\infty \in B_{X(E'')}}\sum\limits_{j=1}^\infty |x_j''(x_j')| \leq 1.
\end{equation}
Calling
$$C_{E''} = \left\{(x_j')_{j=1}^\infty \subseteq E' : \sum_{j=1}^\infty x_j''(x_j') {\rm ~converges~for~every~}(x_j'')_{j=1}^\infty \text{ in } E'' {\rm~such~that~} \right.$$
$$\left. \sup_{(x_j'')_{j=1}^\infty \in B_{X(E'')}} \sum_{j=1}^\infty |x_j''(x_j')|\leq  1 \right\}, $$
it follows from \eqref{eq} that
\begin{align*} \|&(x_j)_{j=1}^\infty  \|_{\left(X^{\rm dual}\right)^{\rm dual}(E)}
=\sup\left\{ \sum_{j=1}^\infty |x_j'(x_j)|: (x_j')_j \in {X^{\rm dual}(E')}, \|(x_j)_{j=1}^\infty\|_{X^{\rm dual}(E')}\leq 1\right\}\\&
= \sup\left\{ \sum_{j=1}^\infty |x_j'(x_j)|: (x_j')_{j=1}^\infty \in C_{E''}\right\} = \sup\left\{ \sum_{j=1}^\infty |J_E(x_j)(x_j')|: (x_j')_{j=1}^\infty \in C_{E''}\right\}\\
&= \|(J_E(x_j))_{j=1}^\infty\|_{X(E'')} \cdot \sup\left\{ \sum_{j=1}^\infty \left|\frac{J_E(x_j)}{ \|(J_E(x_j))_{j=1}^\infty\|_{X(E'')}}(x_j')\right|: (x_j')_{j=1}^\infty \in C_{E''}\right\}\\
& \leq  \|(J_E(x_j))_{j=1}^\infty\|_{X(E'')} \leq \|(x_j)_{j=1}^\infty\|_{X(E)}.
\end{align*}
\end{proof}

\begin{example}\label{exref1}\rm Sometimes the equality $\left(X^{\rm dual}\right)^{\rm dual} = X $ holds and sometimes it does not. On the one hand, $\ell_p(\cdot) = \left(\ell_p^{\rm dual}\right)^{\rm dual}(\cdot)$ for $1 < p < \infty$ by Example \ref{exdual2}(a). On the other hand, if $\left(X^{\rm dual}\right)^{\rm dual} = X $, then $X$ is finitely determined and spherically complete by Proposition \ref{proporp}. In particular, $\left(X^{\rm dual}\right)^{\rm dual} \neq X $ for $X= \ell_p^u(\cdot)$ and $X = {\rm Rad}(\cdot)$.
\end{example}

 The proposition and the examples above motivate the following definition:

\begin{definition}\rm Let $E$ be a Banach space. A sequence class $X$ is said to be {\it $E$-reflexive} if 	$$
	\left(X^{\rm dual}\right)^{\rm dual}(E) = X(E) \text{ and } \|(x_j)_{j=1}^\infty\|_{X(E)} = \|(x_j)_{j=1}^\infty\|_{\left(X^{\rm dual}\right)^{\rm dual}(E)},$$ for every $(x_j)_{j=1}^\infty \in X(E)$.

$X$ is {\it reflexive} if it is $E$-reflexive for every $E$, and $X$ is {\it dual-reflexive} if it is $E'$-reflexive for every $E$.
\end{definition}

\begin{example}\rm We have seen in Example \ref{exref1} that, for $1 < p < \infty$, $\ell_p(\cdot)$ is reflexive. From Example \ref{exdual}(c) and (d) we conclude that, surprisingly, $\ell_1^{\rm dual}(\cdot)$ and $\ell_\infty^{\rm dual}(\cdot)$ are reflexive. 
\end{example}

Now we can prove the main results of this section.

\begin{theorem}\label{classdual}
	Let $X$ and $Y$ be sequence classes and $T \in {\cal L}(E;F)$.\\
{\rm (a)} If $Y$ is linearly stable, finitely dominated and spherically complete, $X$ is dual-representable and $T \in {\cal L}_{X;Y}(E;F)$, then $T' \in {\cal L}_{Y^{\rm dual};X^{\rm dual}}(F';E')$ and $\|T\|_{X;Y} \ge \|T'\|_{Y^{\rm dual};X^{\rm dual}}$.\\
{\rm (b)} If $Y$ is $F$-reflexive and $T' \in {\cal L}_{Y^{\rm dual};X^{\rm dual}}(F';E')$, then $T \in {\cal L}_{X;Y}(E;F)$ and $\|T\|_{X;Y} \le \|T'\|_{Y^{\rm dual};X^{\rm dual}}$.
\end{theorem}

\begin{proof}
(a) Since $T$ is $(X;Y)$-summing and $X$ is dual-representable, from Theorem \ref{dualproof} the composition
$$Y^{\rm dual}(F') \stackrel{J^Y}{\xrightarrow{\hspace{1.1cm}}} Y(F)' \stackrel{(\widehat{T})'}{\xrightarrow{\hspace{1.1cm}}} X(E)' \stackrel{(J^X)^{-1}}{\xrightarrow{\hspace{1.1cm}}} X^{\rm dual}(E')$$
where $J^Y$ and $J^X$ are the corresponding operators as in \eqref{opdual}, gives a well-defined, linear and continuous operator.
All we have to do is to prove that $\widehat{(T')} = (J^X)^{-1} \circ (\widehat{T})' \circ J^Y$: for $(\varphi_j)_{j=1}^\infty \in Y^{\rm dual}(F')$ and $(x_j)_{j=1}^\infty \in X(E)$,
\begin{align}\label{c1}
(\widehat{T})'(J^Y((\varphi_j)_{j=1}^\infty))((x_j)_{j=1}^\infty) & = J^Y((\varphi_j)_{j=1}^\infty)(\widehat{T}((x_j)_{j=1}^\infty))  = J^Y((\varphi_j)_{j=1}^\infty)((T(x_j))_{j=1}^\infty)\nonumber\\
& = \sum_{j=1}^\infty \varphi_j(T(x_j)) = \sum_{j=1}^\infty (\varphi_j \circ T)(x_j). \end{align}
On the other hand, for $\psi \in X(E)'$, writing $(J^X)^{-1}(\psi) = (\xi_j)_{j=1}^\infty$ we have
\begin{equation}\label{c2}
\psi((x_j)_{j=1}^\infty) = \sum_{j=1}^\infty \xi_j(x_j).
\end{equation}
So, from \eqref{c1} and \eqref{c2} it follows that
\[ \left((J^X)^{-1} \circ (\widehat{T})' \circ J^Y\right) ((\varphi_j)_{j=1}^\infty)  = (\varphi_j \circ T)_{j=1}^\infty = (T'(\varphi_j))_{j=1}^\infty = \widehat{(T')}((\varphi_j)_{j=1}^\infty),\]
proving that $T' \in {\cal L}_{Y^{\rm dual};X^{\rm dual}}(F';E')$. Moreover,
\[ \|T'\|_{Y^{\rm dual};X^{\rm dual}} = \left\|\widehat{(T')}\right\| = \left\|(J^X)^{-1} \circ (\widehat{T})' \circ J^Y \right\| \le \left\|(\widehat{T})'\right\|= \left\|\widehat{T}\right\| = \|T\|_{X;Y}. \]

\noindent (b) Let $(x_j)_{j=1}^\infty \in X(E) \stackrel{1}{\hookrightarrow} \left(X^{\rm dual}\right)^{\rm dual}(E)$ be given. Thus, the series $\sum_{j=1}^\infty \psi_j(x_j)$ converges for all $(\psi_j)_{j=1}^\infty \in X^{\rm dual}(E')$. As $T' \in {\cal L}_{Y^{\rm dual};X^{\rm dual}}(F';E')$, we have $(T'(\varphi_j))_{j=1}^\infty \in X^{\rm dual}(E')$ whenever $(\varphi_j)_{j=1}^\infty \in Y^{\rm dual}(F')$ and therefore the series
\[ \sum_{j=1}^\infty \varphi_j(T(x_j)) = \sum_{j=1}^\infty T'(\varphi_j)(x_j)   \]
converges for all $(\varphi_j)_{j=1}^\infty \in Y^{\rm dual}(F')$. This shows that $(T(x_j))_{j=1}^\infty \in (Y^{\rm dual})^{\rm dual}(F)$, hence $(T(x_j))_{j=1}^\infty \in Y(F)$ because $Y$ is $F$-reflexive. It follows that $T \in {\cal L}_{X;Y}(E;F)$. Moreover,
\begin{align*}
\|T\|_{X;Y} & = \|\widehat{T}\| = \sup_{(x_j)_{j=1}^\infty \in B_{X(E)}} \|(T(x_j))_{j=1}^\infty\|_{Y(F)} \\
& = \sup_{(x_j)_{j=1}^\infty \in B_{X(E)}} \sup_{(\varphi_j)_{j=1}^\infty \in  B_{Y^{\rm dual}(F')}} \sum_{j=1}^\infty |\varphi_j(T(x_j))| \\
& = \|T'\|_{Y^{\rm dual};X^{\rm dual}}\cdot \sup_{(x_j)_{j=1}^\infty \in B_{X(E)}} \sup_{(\varphi_j)_{j=1}^\infty \in  B_{Y^{\rm dual}(F')}} \sum_{j=1}^\infty \left|\frac{T'(\varphi_j)}{\|T'\|_{Y^{\rm dual};X^{\rm dual}}}(x_j)\right|\\
& \le \|T'\|_{Y^{\rm dual};X^{\rm dual}} \cdot \sup_{(x_j)_{j=1}^\infty \in B_{X(E)}} \sup_{(\psi_j)_{j=1}^\infty \in  B_{X^{\rm dual}(E')}} \sum_{j=1}^\infty |\psi_j(x_j)| \\
& = \|T'\|_{Y^{\rm dual};X^{\rm dual}}\cdot \sup_{(x_j)_{j=1}^\infty \in B_{X(E)}} \|(x_j)_{j=1}^\infty\|_{X(E)} = \|T'\|_{Y^{\rm dual};X^{\rm dual}}.
\end{align*}
\end{proof}

Now we establish the implications in (\ref{fequi}).

\begin{theorem}\label{classdual2}
	Let $X$ and $Y$ be sequence classes and $T \in {\cal L}(E;F)$.\\
 {\rm (a)} If $X$ is spherically complete, $Y^{\rm dual}$ is dual-representable and dual-reflexive and $T \in {\cal L}_{Y^{\rm dual};X^{\rm dual}}(E;F)$, then $T' \in {\cal L}_{X;Y}(F';E')$ and $\|T\|_{Y^{\rm dual};X^{\rm dual}} \ge \|T'\|_{X;Y}$.\\
 {\rm (b)} If $X$ and $Y$ are spherically complete and $T' \in {\cal L}_{X;Y}(F';E')$, then $T \in {\cal L}_{Y^{\rm dual};X^{\rm dual}}(E;F)$ and $\|T\|_{Y^{\rm dual};X^{\rm dual}} \le \|T'\|_{X;Y}$.
\end{theorem}

\begin{proof}
(a) The operator
$$u \colon X(F') \longrightarrow X^{\rm dual}(F)'~,~u\left((\varphi_j)_{j=1}^\infty \right)\left((y_j)_{j=1}^\infty \right) = \sum_{j=1}^\infty\varphi_j(y_j),  $$
is well defined by the definition of $X^{\rm dual}(F)$, its linearity is obvious and its continuity follows from the definition of $\|\cdot\|_{X^{\rm dual}(F)}$. Since $T'$ is $(X;Y)$-summing and $Y^{\rm dual}$ is dual-representable and dual-reflexive, from Theorem \ref{dualproof} the composition
$$ X(F') \stackrel{u}{\xrightarrow{\hspace{1.3cm}}}  X^{\rm dual}(F)' \stackrel{(\widehat{T})'}{\xrightarrow{\hspace{1.3cm}} } Y^{\rm dual}(E)'\stackrel{\left(J^{Y^{\rm dual}}\right)^{-1}}{\xrightarrow{\hspace{1.3cm}} }  \left(Y^{\rm dual}\right)^{\rm dual}(E')= Y(E')$$
where $J^{Y^{\rm dual}}$ is the corresponding operator as in \eqref{opdual}, gives a well-defined, linear and continuous operator.  Reasoning as in the proof of Theorem \ref{classdual}(a) we get $\widehat{(T')}=\left(J^{Y^{\rm dual}}\right)^{-1} \circ (\widehat{T})' \circ u $, proving that $T' \in {\cal L}_{X;Y}(F';E')$. The norm inequality is straightforward.

\medskip

\noindent (b) Let $(x_j)_{j=1}^\infty \in Y^{\rm dual}(E)$ be given. Thus, the series $\sum_{j=1}^\infty \psi_j(x_j)$ converges for all $(\psi_j)_{j=1}^\infty \in Y(E')$. As $T' \in {\cal L}_{X;Y}(F';E')$, we have $(T'(\varphi_j))_{j=1}^\infty \in Y(E')$ whenever $(\varphi_j)_{j=1}^\infty \in X(F')$ and therefore the series
\[ \sum_{j=1}^\infty \varphi_j(T(x_j)) = \sum_{j=1}^\infty T'(\varphi_j)(x_j)   \]
converges for every $(\varphi_j)_{j=1}^\infty \in X(F')$. This means  that $(T(x_j))_{j=1}^\infty \in X^{\rm dual}(F)$ and so $T \in {\cal L}_{Y^{\rm dual};X^{\rm dual}}(E;F)$. The norm inequality follows as in the proof of Theorem \ref{classdual}(b).
\end{proof}

Let us see that our results recover some well known results and provide new ones.

\begin{example}\rm (a) Take $X = \ell_p^u(\cdot)$ and $Y = \ell_q(\cdot)$, $1 \le p \le q < \infty$. If an operator $T$ is absolutely $(q;p)$-summing, we have by Theorem \ref{classdual}(a) that $T'$ is $(\ell_{q^*}(\cdot);(\ell_p^u)^{\rm dual}(\cdot))$-summing, that is, $T'$ is $(\ell_{q^*}(\cdot);\ell_{p^*}\langle \cdot \rangle)$-summing by Example \ref{exdual2}(c) (remember that $T'$ acts between dual spaces). This recovers a classical result due to Cohen \cite{cohen73}. \\
(b) By Theorem \ref{classdual2}(a), if $ T \in {\cal L}_{\ell_{q^*}(\cdot);(\ell_p^u)^{\rm dual}(\cdot)}(E;F)$, then $T'$ is absolutely $(q;p)$-summing. If $F$ is a dual space, this means that if $ T \in {\cal L}_{\ell_{q^*}(\cdot);\ell_{p^*}\langle\cdot\rangle}(E;F)$, then $T'$ is absolutely $(q;p)$-summing, which recovers another classical result from  \cite{cohen73}.\\
(c) If $T$ is such that $T'$ is absolutely $(q;p)$-summing, then $T$ is $(\ell_{q^*}(\cdot); (\ell_p^u)^{\rm dual}(\cdot))$-summing by Theorem \ref{classdual2}(b). If the target space is a dual space, this means that if $T'$ is absolutely $(q;p)$-summing then $T$ is $(\ell_{q^*}(\cdot); \ell_{p^*}\langle\cdot\rangle)$-summing, recovering a third classical result from   \cite{cohen73}. \\
(d) As to new implications, we give an illustrative example. According to \cite{botelhocampossantos}, an operator is absolutely mid-$(p,q)$-summing, $1 \leq p < q < \infty$, if it is $(\ell_q^{\rm mid}(\cdot); \ell_p(\cdot))$-summing. By Theorem \ref{classdual2}(b), if the adjoint $T'$ of an operator $T$ is $(\ell_{p}^{\rm mid}(\cdot);\ell_q(\cdot))$-summing, then $T$ is $(\ell_{q^*}(\cdot); \ell_p^{\rm mid}\left\langle\cdot\right\rangle)$-summing.
\end{example}

From now on $X$ and $Y$ are sequence classes such that ${\cal L}_{X;Y}$ is a Banach operator ideal (cf. \cite[Theorem 3.6]{botelhocampos}). The following lemma is a straightforward consequence of \cite[Proposition 2.4 and Corollary 2.6]{botelhocampos}.

\begin{lemma}\label{llema} Let $X$ and $Y$ be finitely dominated sequence classes with $Y$ finitely injective. Then the Banach operator ideal ${\cal L}_{X;Y}$ is injective.
\end{lemma}

\begin{corollary}\label{bidual} Let $X$ and $Y$ be sequence classes with $X$ dual-representable, $Y$ spherically complete and dual-reflexive such that $Y^{\rm dual}$ is dual-representable. For an operator $T \colon E \longrightarrow F$, $T \in {\cal L}_{X;Y}(E;F)$ if and only if $T'' \in {\cal L}_{X;Y}(E'';F'')$. In this case, $\|T\|_{X;Y} = \|T''\|_{X;Y}$.
\end{corollary}

\begin{proof} Let $T \in {\cal L}_{X;Y}(E;F)$. Then $T' \in {\cal L}_{Y^{\rm dual}; X^{\rm dual}}(F';E')$ by Theorem \ref{classdual}(a), so  $T''\in {\cal L}_{(X^{\rm dual)^{\rm dual}}; (Y^{\rm dual)^{\rm dual}}}(E'';F'')$ by Theorem \ref{classdual2}(a). It follows that $T'' \in {\cal L}_{X;Y}(E'';F'')$ by Propositon \ref{canemb} combined with the dual-reflexivity of $Y$. The corresponding norm inequality follows accordingly. The reverse implication/norm inequality holds for injective operator ideals in general, so Lemma \ref{llema} completes the proof.
\end{proof}

\begin{remark}\rm Of course, the corollary above could be pursued searching conditions under which the ideal ${\cal L}_{X;Y}$ is maximal. But, in this case, we would not have the relationships we established for $T$ and $T'$.
\end{remark}

\begin{example}\rm Taking $1 \leq p < \infty$, $X = \ell_p^u(\cdot)$ and $Y = \ell_p(\cdot)$, Corollary \ref{bidual} recovers the following classical equivalence: an operator is absolutely $p$-summing if and only if its second adjoint is absolutely $p$-summing \cite[Proposition 2.19]{djt}.
\end{example}

\bigskip

\noindent Faculdade de Matem\'atica~~~~~~~~~~~~~~~~~~~~~~Departamento de Ci\^{e}ncias Exatas\\
Universidade Federal de Uberl\^andia~~~~~~~~ Universidade Federal da Para\'iba\\
38.400-902 -- Uberl\^andia -- Brazil~~~~~~~~~~~~ 58.297-000 -- Rio Tinto -- Brazil\\
e-mail: botelho@ufu.br \hspace*{5,7cm} and

\noindent \hspace*{7,7cm}Departamento de Matem\'atica\\
\hspace*{7,7cm}Universidade Federal da Para\'iba\\
\hspace*{7,7cm}58.051-900 -- Jo\~ao Pessoa -- Brazil

\noindent \hspace*{7,7cm}e-mails: jamilson@dcx.ufpb.br,\\
\hspace*{9,3cm}jamilsonrc@gmail.com

\end{document}